\documentclass[12pt]{article}  

\usepackage{amsmath,amssymb,amsthm}
\usepackage[latin1]{inputenc}

\newtheorem{theorem}{Theorem}[section]
\newtheorem{lemma}[theorem]{Lemma}
\newtheorem{proposition}[theorem]{Proposition}
\newtheorem{remark}[theorem]{Remark}
\theoremstyle{corollary}
\newtheorem{corollary}[theorem]{Corollary}

\def\N{\mathbb N}
\def\C{\mathbb C}
\def\R{\mathbb R}
\def\D{\mathbb D}

\def\Q{\mathbb Q}

\def\ve{\varepsilon}

\def\al{\alpha}
\def\la{\lambda}

\def\ovl{\overline}

\date{}

\begin{document}

\title{Vector spaces of non-extendable holomorphic functions}

\author{Luis Bernal-Gonz\'alez}

\maketitle

{\footnotesize  
{\sl
\centerline{Departamento de An\'alisis Matem\'atico. Facultad de Matem\'aticas.}
\centerline{Apdo.~1160. Avda.~Reina Mercedes, 41080 Sevilla, Spain.}
\centerline{E-mail: {\tt lbernal@us.es}}
}}

\vskip 10pt

\centerline{\sl To Professor Jos\'e Bonet Solves on his 60th birthday}

\begin{abstract}
\noindent In this paper, the linear structure of the family $H_e(G)$ of holomorphic functions in a domain $G$ of the complex plane that are not
analytically continuable beyond the boundary of $G$ is analyzed. We prove that $H_e(G)$ contains, except for zero, a dense algebra; and, under appropriate conditions, the subfamily of $H_e(G)$ consisting of boundary-regular functions contains dense vector spaces with maximal dimension,
as well as infinite dimensional closed vector spaces and large algebras. The case in which $G$ is a domain of existence in a complex Banach space is also considered. The results obtained complete or extend a number of previous ones by several authors.

\vskip .15cm


\vskip .15cm

\end{abstract}

\section{Introduction}

\quad This paper intends to be a contribution to the study of the linear structure of the family of
non-extendable holomorphic functions. The search of linear (or, in general, algebraic) structures
within nonlinear sets has become a trend in the last two decades, see e.g.~the survey \cite{BPS}.
Here we restrict ourselves to the setting of complex analytic functions, with focus on those ones
that cannot be continued beyond the boundary of the domain.

\vskip .15cm

Although our main concern is the complex plane $\C$, it is convenient to state definitions in a more
general framework. Assume that $E$ is a complex Banach space. A {\it domain} $G$ in $E$ is a nonempty
connected open subset of $E$. Along this paper, we will assume that $G \ne E$.
Denote by $H(G)$ the space of all holomorphic functions $f:G \to \C$
(see e.g.~\cite{Chae} for definitions and pro\-per\-ties), and by $\partial G$ the boundary of $G$. We say
that a function $f \in H(G)$ is {\it holomorphically non-extendable across any boundary point}
(synonymous expressions are: $f$ is {\it analytically non-continuable beyond} $\partial G$, $f$ is
{\it holomorphic exactly} on $G$, $G$ is a {\it domain of existence} of $f$) whenever there do not exist
two domains $G_1$ and $G_2$ in $E$ and $\widetilde{f} \in H(G_1)$ such that $G_2 \subset G \cap G_1$,
$G_1 \not\subset G$ and $\widetilde{f} = f$ on $G_2$. We denote by $H_e(G)$ the family of all $f \in H(G)$
that are holomorphic exactly on $G$. A domain $G \subset E$ is said to be a {\it domain of existence} if
it is a domain of existence of some function $f \in H(G)$ (that is, if $H_e(G) \ne \emptyset$). And $G$
is called a {\it domain of holomorphy} \,provided that there do not exist
two domains $G_1$ and $G_2$ in $E$ with $G_2 \subset G \cap G_1$,
$G_1 \not\subset G$ such that, for every $f \in H(G)$, there exists $\widetilde{f} \in H(G_1)$ with $\widetilde{f} = f$ on $G_2$.

\vskip .15cm

Plainly, every domain of existence is a domain of holomorphy. In the case where
$E = \C^N$ $(N \in \N := \{1,2,...\})$, the Cartan--Thullen theorem \cite{Kaup} asserts that
$G$ is a domain of existence if and only if
$G$ is a domain of holomorphy, and if and only if
$G$ is holomorphically convex, that is, for every compact subset $K$ of $G$, the set $\widehat{K} :=
\{x \in G: \, |f(x)| \le \sup_K |f|$ for all $f \in H(G)\}$ satisfies dist$(\widehat{K},E \setminus G) > 0$.

\vskip .15cm

Turning to the case $E = \C$, in 1884 Mittag-Leffler proved that every domain $G \subset \C$ is a domain of
existence \cite[Chapter 10]{Hille} (this is no longer true for
higher dimensions, see e.g.~\cite{Krantz}). Moreover, $f \in H_e(G)$ if and only if $\rho (f,a) = {\rm dist}(a,\partial G)$
for all $a \in G$, where $\rho (f,a)$ denotes the radius of convergence of the Taylor series of $f$
with center at $a$. Of course, we have $H_e(G) \subset H_{we}(G)$, where
\,$H_{we} (G)$ \,stands for the class of functions which are holomorphic
weakly exactly on $G$, that is, $f \in H_{we}(G)$ \,if and only if $f$ has no holomorphic extension to any
domain containing $G$ strictly. But the reverse inclusion is not true:
take e.g.~$G = \C \setminus (-\infty ,0]$
and $f :=$ the principal branch of $\log z$. Observe that \,$H_e(G) = H_{we}(G)$
\,provided that \,$G$ \,is a {\it Jordan domain}, i.e.~a domain in $\C$ such that $\partial G$
is a homeomorphic image in $\C_\infty$ of a circle. Here $\C_\infty := \C \,\cup \, \{\infty\}$
denotes the one-point compactification of $\C$. Note that we allow unbounded domains: for instance,
an open half-plane is Jordan.

\vskip .15cm

Recall that a domain \,$G \subset \C$ \,is said to be {\it regular} if
\,$G = \overline{G}^0$ ($A^0$ denotes the interior of $A$, while $\overline{A}$ stands for the closure of $A$).
It is plain that every Jordan domain is regular, but there are regular (even simply connected) domains that are not Jordan,
for instance, $G = \{z: |z-1| < 1$ and $|z-(1/2)| > 1/2\}$.

\vskip .15cm

For every domain $G$ of a complex Banach space $E$, the space $H(G)$ will be endowed with the topology
of uniform convergence on compacta. If $E = \C^N$ $(N \in \N )$, $H(G)$ becomes a Fr\'echet space
(i.e.~a complete metrizable locally convex space), but it is no longer metrizable if $E$ is
infinite dimensional, see \cite{Alex,AnsP,Chae}. In 1933 Kierst and Szpilrajn \cite{KiS} showed in the case of
the open unit disc $G = \D = \{z \in \C : \, |z| < 1\}$ that the property discovered by Mittag-Leffler
is topologically generic; specifically, $H_e(\D)$ is residual (so dense) in $H(\D )$, that is, its complement
in $H(\D )$ is of first category. For extensions and improvements of the Mittag-Leffler and Kierst--Szpilrajn
theorems, see \cite{Bieber,Jarni,Kaha68,Lelong,Ryll}. Kahane and the author
(see \cite[Theorem 3.1 and following remarks]{Kah} and \cite[Theorem 3.1]{BerS}) observed that
the residuality of $H_e(G)$ holds for many subspaces $X$ of $H(G)$:

\begin{theorem} \label{Th Kahane}
Let $G \subset \C$ be a domain and $X$ be a Baire topological
vector space with $X \subset H(G)$ such that all evaluation functionals
$$
f \in X \mapsto
f^{(k)} (a) \in \C \,\,\,\, (a \in G, \, k \ge 0)
$$
are con\-ti\-nuous and,
for every  $a \in G$ and every  $r > {\rm dist}\,(a,\partial G)$,
there exists  $f \in X$  such that  $\rho (f,a) < r$.
Then $H_e(G) \cap X$ is residual in $X$.
\end{theorem}

Of course the case $X = H(G)$ is included. But Theorem \ref{Th Kahane} also includes some interesting
strict subspaces, such as Hardy ($H^p$, $p>0$) and Bergman ($B^p$, $p>0$) spaces (in the case
$G = \D$, see \cite{BerS}; for definitions and properties, see e.g.~\cite{Zhu})
and the space $A^\infty (G)$ (see \cite{BerCL}) if $G$ is a {\it regular}
domain. By $A^\infty (G)$ we have denoted the class of {\it boundary-regular} holomorphic functions in $G$,
that is, $A^\infty (G) = \{ f \in H(G): \ f^{(j)} $ has a continuous extension to $\overline{G}$ for all
$j \in \N_0 \} $, where $\N_0 := \{0,1,2, \dots \}$. It becomes a Fr\'echet space when it is endowed with the topology of uniform convergence
of functions and all their derivatives on each compact set $K \subset \overline{G}$.
Chmielowski \cite{Chm} had established in 1980 that
$H_e(G) \cap A^\infty (G) \ne \emptyset$ if $G$ is regular; see also \cite{Siciak}.

\vskip .15cm

In Section 2 we will recall some lineability notions and
review the main known results about the algebraic structure of $H_e(G)$ in $H(G)$ and in subspaces of it,
including the infinite dimensional case. Section 3 contains our new statements on the linear structure
--in its diverse degrees-- of $H_e(G)$, with special emphasis on the class of boundary-regular holomorphic functions.

\section{Lineability notions and known results}

\quad When a set $A$ lives in a bigger set $X$ endowed with some structure (vector space, topological vector space, algebra),
an alternative way to measure the size of $A$ involves finding large sub-structures within $A$.
A number of concepts have been coined in order to describe the
algebraic size of a set, see  \cite{AGS,Bay1,Ber2,GuQ}
(see also \cite{BPS} for an account of lineability properties of specific subsets of vector spaces).
Namely, if $X$ is a vector space, $\alpha$ is a cardinal number and $A \subset X$, then $A$ is said to be:
\begin{enumerate}
\item[$\bullet$] {\it lineable} if there is an infinite dimensional vector space $M$ such that $M \setminus \{0\} \subset A$,
\item[$\bullet$] {\it $\alpha$-lineable} if there exists a vector space $M$ with dim$(M) = \alpha$ and $M \setminus \{0\} \subset A$
(hence lineability means $\aleph_0$-lineability,
where $\aleph_0 = {\rm card}\,(\N )$, the cardinality of the set of positive integers), and
\item[$\bullet$] {\it maximal lineable} in $X$ if $A$ is ${\rm dim}\,(X)$-lineable.
\end{enumerate}
If, in addition, $X$ is a topological vector space, then $A$ is said to be:
\begin{enumerate}
\item[$\bullet$] {\it dense-lineable} or {\it algebraically generic} in $X$
whenever there is a dense vector subspace $M$ of $X$ satisfying $M \setminus \{0\} \subset A$
(hence dense-lineability implies lineability as soon as dim$(X) = \infty$),
\item[$\bullet$] {\it maximal dense-lineable} in $X$
whenever there is a dense vector subspace $M$ of $X$ satisfying $M \setminus \{0\} \subset A$ and
dim$\,(M) =$ dim$\,(X)$, and
\item[$\bullet$] {\it spaceable} in $X$ if there is a closed infinite dimensional vector
subspace $M$ such that $M \setminus \{0\} \subset A$
(hence spaceability implies lineability).
\end{enumerate}
And, according to \cite{APS,BarG}, when $X$ is a topological vector space contained in some (linear) algebra then
$A$ is called:
\begin{enumerate}
\item[$\bullet$] {\it algebrable} if there is an algebra \,$M$ so
    that $M \setminus \{0\} \subset A$ and $M$ is infinitely generated, that is, the cardinality of any system of generators of \,$M$ is infinite.
\item[$\bullet$] {\it densely algebrable} in $X$ if, in addition, $M$ can be taken dense in $X$.
\item[$\bullet$] {\it $\alpha$-algebrable} if there is an $\alpha$-generated algebra \,$M$ with \,$M \setminus \{0\} \subset A$.
\item[$\bullet$] {\it strongly $\alpha$-algebrable} if there exists an $\alpha$-generated {\it free} algebra \,$M$ with \,$M \setminus \{0\} \subset A$ (for $\alpha = \aleph_0$, we simply say {\it strongly algebrable}).
\item[$\bullet$] {\it densely strongly $\alpha$-algebrable} if, in addition, the free algebra \,$M$ can be taken dense in $X$.
\end{enumerate}
Note that if $X$ is contained in a commutative algebra then a set $B \subset X$ is a generating set of some free algebra contained in $A$
if and only if for any $N \in \N$, any nonzero polynomial $P$ in $N$ variables without constant term and any distinct $f_1,...,f_N \in B$, we have
$P(f_1,...,f_N) \ne 0$ and $P(f_1,...,f_N) \in A$. Observe that strong $\alpha$-algebrability $\Rightarrow$ $\alpha$-algebrability $\Rightarrow$ $\alpha$-lineability, and none of these implications can be reversed, see \cite[p.~74]{BPS}.

\vskip .15cm

The next dense-lineability criterion, that can be found in \cite[Theorem 2.3]{BerOrd} and is an extension of statements from \cite{AGPS,Ber,Ber2}, will be used later.

\begin{theorem}\label{maximal dense-lineable}
Assume that \,$X$ is a topological vector space. Let $A \subset X$ be an $\al$-lineable subset.
Suppose that there exists a
dense-lineable subset $B \subset X$ such that $A + B \subset A$, and that
$X$ has an open basis \,$\cal B$ for its to\-po\-lo\-gy
such that ${\rm card} ({\cal B}) \le \al$. Then $A$ is dense-lineable and if, in addition, $A \cap B = \emptyset$,
then $A \cup \{0\}$ contains a dense vector space $D$ with \,${\rm dim} (D) = \al$.
\end{theorem}

As for spaceability, we provide the following statement that is ascribed by Kitson and Timoney \cite{KitT} to Kalton.
It is, in turn, a Fr\'echet version of an earlier result by Wilansky \cite{Wil} given in the Banach setting.
Theorem \ref{Wilansky-Kalton} will be needed in the proof of Theorem \ref{H_e(G)Ainfty spaceable} below.

\begin{theorem} \label{Wilansky-Kalton}
If $X$ is a Fr\'echet space and $Y \subset X$ is a closed linear subspace, then
the complement $X \setminus Y$ is spaceable if and only if \,$Y$ has infinite codimension.
\end{theorem}

Concerning algebrability, the following criterion is given in \cite[Proposition 7]{BalBF} (see also \cite[Theorem 1.5]{BBFG}) for
a family of functions ${\cal F} \subset \R^{[0,1]}$. By mimicking its proof, in Proposition \ref{exponentials} below we provide an extension
to the case ${\cal F} \subset \C^\Omega$, which will be needed in Section 3. By $\cal E$ we denote the family of exponential-like functions
$\C \to \C$, that is, the functions of the form \,$\varphi (z) = \sum_{j=1}^m a_j e^{b_j z}$ \,for some $m \in \N$, some $a_1,...,a_m \in \C \setminus \{0\}$ and some distinct $b_1,...,b_m \in \C \setminus \{0\}$. As usual, $\mathfrak{c}$ will stand for the cardinality of the continuum.

\begin{proposition} \label{exponentials}
Let \,$\Omega$ be a nonempty set and ${\cal F} \subset \C^\Omega$. Assume that there exists a function \,$f \in {\cal F}$
such that \,$f(\Omega )$ is uncountable and \,$\varphi \circ f \in {\cal F}$ \,for every $\varphi \in {\cal E}$.
Then \,${\cal F}$ is strongly $\mathfrak{c}$-algebrable. More precisely, if $H \subset (0,+\infty )$ is a set with
\,{\rm card}$(H) = \mathfrak{c}$ \,and linearly independent over the field $\Q$ of rational numbers, then
$$
\{\exp \circ (rf): \, r \in H\}
$$
is a free system of generators of an algebra contained in \,${\cal F} \cup \{0\}$.
\end{proposition}

\begin{proof}
Firstly, each function \,$\varphi (z) = \sum_{j=1}^m a_j e^{b_j z}$ in ${\cal E}$ (with $a_1, \dots ,b_m \in \C \setminus \{0\}$
and $b_1,...,b_m$ distinct) has at most countably many zeros. Indeed, we can assume $|b_1| = \cdots = |b_p| > |b_j|$ $(j=p+1,...,m)$.
Then $b_j = |b_1| e^{i \theta_j}$ $(j=1,...,p)$ with $|\theta_j - \theta_1| \in (0,\pi ]$ for $j=2,...,p$
(so $c_j := \cos (\theta_j - \theta_1) < 1$ for $j=2,...,p$). Hence we have for all $r > 0$ that
$$
|\varphi (re^{-i \theta_1})| \ge |a_1| e^{|b_1|r} - \sum_{j=2}^p |a_j| e^{|b_1|c_j r} - \sum_{j=p+1}^m |a_j| e^{|b_j|r} \longrightarrow +\infty
\,\,\, \hbox{as} \,\, r \to +\infty .
$$
Therefore $\varphi$ is a nonconstant entire function, so $\varphi^{-1}(\{0\})$ is countable. Now, consider a
nonzero polynomial $P$ in $N$ complex variables without constant term, as well as numbers $r_1, \dots ,r_N \in H$.
The function $\Phi : \Omega \to \C$ given by $\Phi = P(\exp \circ (r_1f), \dots ,\exp \circ (r_Nf)) \in \C$ is of the form
$$
\sum_{j=1}^m a_i (e^{r_1f(x)})^{k(j,1)} \cdots (e^{r_Nf(x)})^{k(j,N)} = \sum_{j=1}^m a_i \exp \left(f(x) \sum_{l=1}^N r_l k(j,l)\right),
$$
where $a_1, \dots ,a_m \in \C \setminus \{0\}$ and the matrix $[k(j,l)]_{j=1,...,m \atop l=1,...,N}$ of nonnegative integers has
distinct nonzero rows. Thus, the numbers $b_j := \sum_{l=1}^N r_l k(j,l)$ $(j=1,...,m)$ are distinct and nonzero; hence the function
\,$\varphi (z) :=
\sum_{j=1}^m a_j e^{b_j z}$ \,belongs to \,$\cal E$. But $\Phi = \varphi \circ f$, so if $\Phi \equiv 0$ then we would have
$\varphi |_{f(\Omega )} \equiv 0$, which contradicts the fact that $f(\Omega )$ is uncountable. Consequently, $\Phi \ne 0$ and,
by hypothesis, $\Phi \in {\cal F}$. This proves the proposition.
\end{proof}

Turning to our setting of non-extendable holomorphic functions, and using the previous terminology,
Aron, Garc\'{\i}a and Maestre \cite{AGM} proved in 2001 the following.

\begin{theorem} \label{AronGarciaMaestre}
Assume that $N \in \N$ and $G \subset \C^N$ is a domain of existence. Then $H_e(G)$ is dense-lineable, spaceable and algebrable.
In fact, there is a closed infinitely generated algebra contained in $H_e(G)$.
\end{theorem}

Recall that a subset $A$ of a locally convex space $E$ is said to be {\it sum-absorbing}
whenever there is $\la > 0$ such that $\la (A + A) \subset A$, and $E$ is called {\it nearly-Baire}
if, given a sequence $(A_j)$ of sum-absorbing balanced closed subsets with $E = \bigcup_{j=1}^\infty A_j$,
there is $j_0$ such that $A_{j_0}$ is a neighborhood of $0$. In 2008, Valdivia \cite{Val} showed that
the dense subspace contained in $H_e(G) \cup \{0\}$ can be chosen to be nearly-Baire for any domain
of existence $G \subset \C^N$. In the case $N=1$, the author had demonstrated in 2006 \cite{BerB} that
$H_e(G)$ is {\it maximal dense-lineable} in $H(G)$ for any domain $G \subset \C$.

\vskip .15cm

In the special case $G = \D$, Aron {\it et al.}~\cite{AGM} considered the nonseparable Banach space
$H^\infty := \{f \in H(\D ): \, f$ is bounded on $\D\}$, endowed with the supremum norm, and showed
that $H_e(\D ) \cap H^\infty$ contains, except for zero, an infinitely generated algebra that is nonseparable
and closed in $H^\infty$. The author \cite{BerS} obtained in 2005 that, under appropriate conditions
on a function space $X \subset H(G)$, the set $H_e(\D) \cap X$ is dense-lineable or spaceable in $X$.
In particular, the families $H_e(\D ) \cap H^p , \, H_e(\D ) \cap B^p$ $(p>0)$ and $H_e(\D ) \cap A^\infty (\D )$
turn out to be dense-lineable as well as spaceable in $H^p, \, B^p$ and $A^\infty (\D )$, respectively.

\vskip .15cm

We say that a domain $G \subset \C$ is {\it finite-length} provided that there is $M \in (0,+\infty )$ such that
for any pair $a,b \in G$ there exists a curve $\gamma \subset G$ joining \,$a$ \,to \,$b$ \,for which ${\rm length} (\gamma ) \le M$.
In 2008, Bernal {\it et al.}~\cite{BerCL} established the following assertion, showing that regularity plus appropriate
metrical and topological conditions assure dense-lineability for $H_e(G)$ in $A^\infty (G)$.

\begin{theorem} \label{Bernal-MCCM-Luh}
Let \,$G \subset \C$ be a finite-length regular domain such that \,$\C \setminus \overline{G}$
\,is connected. Then \,$H_e(G) \cap A^\infty (G)$ \,is dense-lineable in \,$A^\infty (G)$.
\end{theorem}

And Valdivia \cite{Val2} obtained in 2009
the following more precise result for the bigger class \,$H_{we}(G)$ \,under less restrictive assumptions.

\begin{theorem} \label{Valdivia}
Let $G \subset \C$ be a regular domain. Then there exists a nearly-Baire dense subspace \,$M \subset A^\infty (G)$ \,such
that \,$M \subset A^\infty (G)$ \,and \,$M \setminus \{0\} \subset H_{we}(G)$.
\end{theorem}

The same result still holds if we replace \,$A^\infty (G)$ \,by the smaller space
$$
A^\infty_b(G) := \{f \in A^\infty (G): \hbox{ each } f^{(j)} \,\, (j=0,1, \dots ) \,\hbox{ is bounded on } G\},
$$
see \cite{Val2}. Note that, as a consequence of Theorem \ref{Valdivia}, we obtain the following.

\begin{corollary}
If \,$G \subset \C$ \,is a Jordan domain then the family \,$H_e(G) \cap A^\infty (G)$ \,is dense-lineable.
\end{corollary}

We have, in addition, the next theorem, whose parts (a) and (b) are showed in
\cite{BerCL}, while part (c) is proved in \cite{BerOrd} by using Theorem \ref{maximal dense-lineable} above.

\begin{theorem} \label{Jordan-X-finitelength}
Assume that $G \subset \C$ is a domain. We have:
\begin{enumerate}
\item[\rm (a)] If \,$G$ is a Jordan domain with analytic boundary then $H_e(G) \cap A^\infty (G)$ is spaceable in $A^\infty (G)$.
\item[\rm (b)] If \,$\partial G$ does not contain isolated points, $X \subset H(G)$ is a vector space with \,$H_e(G) \cap X \ne \emptyset$
\,and there is a nonconstant function \,$\varphi$ \,holomorphic on some domain \,$\Omega \supset \ovl{G}$ \,such that
\,$\varphi X \subset X$, then \,$H_e(G) \cap X$ \,is li\-ne\-a\-ble.
\item[\rm (c)] If \,$G$ \,is a regular finite-length domain such that \,$\C \setminus \ovl{G}$ \,is connected then
\,$H_e(G) \cap A^\infty (G)$ \,is maximal dense-lineable in \,$A^\infty (G)$.
\end{enumerate}
\end{theorem}

Finally, in the infinite dimensional setting, Alves \cite{Alv} has recently proved the following assertion.

\begin{theorem} \label{Alves}
Suppose that \,$G$ is a domain of existence of a separable complex Banach space $E$. Then \,$H_e(G)$ is $\mathfrak{c}$-lineable
and algebrable.
\end{theorem}

In particular, $H_e(G)$ is maximal lineable in $H(G)$. The proof in \cite{Alv} yields in fact the strong algebrability of \,$H_e(G)$.

\section{Main results}

\quad We start with a theorem that, in the case $E = \C$, complements Theorems \ref{AronGarciaMaestre} and \ref{Alves}.
The following algebraic concept is in order. Let $p \in \N_0^p$ and consider the {\it lexicographical order}
``$\le$'' on $\N^p$ defined by: $M := (m_1,...,m_p) \le (j_1,...,j_p) =: J$ if and only if $M=J$ or $M < J$; and $M < J$
if and only if there is $s \in \{1,...,p\}$ such that $m_k = j_k$ for $k \le s-1$ and $m_s < j_s$.
Since it is a total order on $\N_0^p$, every nonempty finite subset $S \subset \N_0^p$ reaches a maximum $R = (r_1,...,r_p)$.
Denote by ${\cal P}_{p,0}$ the family of all nonzero polynomials of $p$ complex variables without constant term.
Given $P \in {\cal P}_{p,0}$, there is a nonempty finite set $S \subset \N_0 \setminus \{(0,...,0)\}$
such that $P(z_1, \dots ,z_p) = \sum_{J \in S} c_J z_1^{j_1} \cdots z_p^{j^p}$ and $c_J \in \C \setminus \{0\}$
for all $J \in S$.
If $R = \max S$ then we say that $c_R z_1^{r_1} \cdots z_p^{r_p}$ is the {\it dominant monomial} of $P$.

\vskip .15cm

As usual, the Euclidean open ball with center $a \in \C$ and radius $r > 0$ will be denoted by $B(a,r)$.
Moreover, $C(A)$ will stand for the set of all continuous functions $A \to \C$, where $A \subset \C$.

\begin{theorem} \label{H_e(G)densely strongly c-algebrable}
For any domain $G \subset \C$, the set $H_e(G)$ is densely strongly $\mathfrak{c}$-algebrable in $H(G)$.
\end{theorem}

\begin{proof}
We need some topological preparation before constructing an adequate free algebra.
We denote by $G_*$ the one-point compactification of
$G$. Recall that in $G_*$ the whole boundary $\partial G$
collapses to a unique point, say $\omega$. Let us fix an
increasing sequence $\{K_N:\, N \in \N\}$ of compact subsets of
$G$ such that each compact subset of $G$ is contained in some
$K_N$ and each connected component of the complement of every
$K_N$ contains some connected component of the complement of $G$,
see \cite[Chapter 7]{Conway}. Choose a countable dense subset
$\{g_N: \, N \in \N\}$ of the (separable) space $H(G)$.

\vskip .15cm

Select also a sequence $\{a_n: \, n \in \N\}$ of distinct points
of $G$ such that it has no accumulation point in $G$ and each
{\it prime end} (see \cite[Chapter 9]{ColLow}) of $\partial G$ is an
accumulation point of the sequence. More precisely, the sequence
$\{a_n\}_{n \ge 1}$ should have the following property: for every $a \in G$
and every $r >$ dist$(a,\partial G)$, the intersection of
$\{a_n\}_{n \ge 1}$ with the connected component of $B(a,r) \cap G$
containing $a$ is infinite. An example of the required sequence
may be defined as follows. Let $A = \{\alpha_k\}_{k \ge 1}$ be a dense
countable subset of $G$. For each $k \in \N$ choose $b_k \in
\partial G$ such that $|b_k - \alpha_k| =$ dist$(\alpha_k,\partial
G)$. For every $k \in \N$ let $\{a_{k,l}: \, l \in \N\}$ be a
sequence of points of the line interval joining $\alpha_k$ with
the corresponding point $b_k$ such that $|a_{k,l} - b_k| < 1/(k +
l)$ $(k,l \in \N )$. Each one-fold sequence $\{a_n\}$ (without
repetitions) consisting of all distinct points of the set
$\{a_{k,l}: \, k,l \in \N\}$ has the required property.

\vskip .15cm

Fix $N \in \N$. For the set $A_N := K_N \cup
\{a_n: \, n \in \N\}$ we have:
\begin{enumerate}
\item[$\bullet$] The set $A_N$ is closed in $G$ because the set
                 $\{a_n: \, n \in \N\}$ does not cluster in $G$.
\item[$\bullet$] The set $G_* \setminus A_N$ is connected due to
the
                 shape of $K_N$ (recall that in $G_*$ the whole
                 boundary $\partial G$ collapses to $\omega$) and to the denumerability of
                 $\{a_n: \, n \in \N\}$.
\item[$\bullet$] The set $G_* \setminus A_N$ is locally connected
                 at $\omega$, again by the denu\-me\-ra\-bility of
                 $\{a_n: \, n \in \N\}$ and by the fact that one
                 can suppose that neighborhoods of \,$\omega$ \,do not
                 intersect $K_N$.
\end{enumerate}
In other words, each $A_N$ is an Arakelian subset of $G$, see \cite{Gaier1}.
Now, we define a family of functions $B = \{f_\al : \, \al \ge 1\} \subset H(G)$ as follows.
If $\al \not\in \N$, by the Weierstrass interpolation theorem one can select $f_\al \in H(G)$
such that
$$f_\al (a_n) = e^{n^\al} \,\,\, (n=1,2, \dots ),$$
because $\{a_n\}_{n \ge 1}$ lacks accumulation
points in $G$ (see e.g.~\cite[Chapter 13]{Rudin}). Let $\al = N \in \N$. Consider the function $h_N:A_N \to \C$ given by
$$
h_N(z) =  \left\{
\begin{array}{ll}
                 g_N(z)  & \mbox{if } z \in K_N  \\
                 e^{n^N} & \mbox{if } z = a_n \mbox{ and } a_n \not\in K_N.
\end{array} \right.
$$
Note that $h_N \in C(A_N) \cap H(A_N^0)$, $A_N$ is Arakelian and $a_n \in A_N \setminus \ovl{A_N^0}$
whenever \,$a_n \not\in K_N$. Under these conditions, a remarkable approximation-interpolation result due to
Gauthier and Hengartner \cite[Theorem and remark 2 in page 702]{GauH} asserts the existence of
a function $f_N \in H(G)$ such that
$$|f_N(z) - h_N(z)| < {1 \over N} \, \hbox{ for all } z \in A_N, \hbox{ and}$$
$$f_N(a_n) = h_N(a_n) \, \hbox{ for all } n \in \N \, \hbox{ with } a_n \not\in K_N$$
In particular,
$$|f_N(z) - g_N(z)| < 1/N \quad (z \in K_N) \eqno (1)$$
and $f_N(a_n) = e^{n^N}$ provided that $a_n \notin K_N$.
Denote by $\cal A$ the algebra ge\-ne\-ra\-ted by $B$. Since each compact set $K \subset G$ is contained in all $K_N$'s
(except for a finite number of them) and, from (1), we have $\sup_{z \in K} |f_N(z)-g_N(z)| \to 0$ as $N \to \infty$,
the density of $\{g_N\}_{N \ge 1}$ forces $\{f_N\}_{N \ge 1}$ to be dense, so $\cal A$ is dense.

\vskip .15cm

Finally, we prove that \,$\cal A$ \,is freely $\mathfrak{c}$-generated and \,${\cal A} \setminus \{0\} \subset H_e(G)$.
For this, observe that, of course, card$\,[1,+\infty ) = \mathfrak{c}$, and fix $p \in \N$
as well as a polynomial
$$
P(z_1, \dots ,z_p) = \sum_{J \in S} c_J z_1^{j_1} \cdots z_p^{j_p} \in {\cal P}_{p,0},
$$
with its shape described at the beginning of this section.
Let $c_R z_1^{r_1} \cdots z_p^{r_p}$ be its dominant monomial.
Also, let $\al_1, \dots ,\al_p$ be different numbers of $[1,+\infty )$. We can assume
$\al_1 > \al_2 > \cdots > \al_p$. Suppose that $P(f_{\al_1}, \dots ,f_{\al_p}) \equiv 0$. If $N = [\alpha_1]$
(the integer part of $\al_1$) then, since $K_N$ is compact in $G$ and $\{a_n\}_{n \ge 1}$ lacks
accumulation points in $G$, there is $n_0 \in \N$ for which $a_n \notin K_N$ whenever $n \ge n_0$.
Hence $f_{\al_j}(a_n) = e^{n^{\al_j}}$ for all $j \in \{1,...,p\}$ and all $n \ge n_0$.
Now, observe that, for $n \ge n_0$, $P(f_{\al_1}(a_n), \dots ,f_{\al_p}(a_n))$ is a sum of one term of the form
\,$D_n = c_R e^{r_1 n^{\al_1} + \cdots + r_p n^{\al_p}}$
and finitely many terms of the form \,$E_n = c_J e^{j_1 n^{\al_1} + \cdots + j_p n^{\al_p}}$.
The definition of dominant monomial and the assumption $\al_1 > \al_2 > \cdots > \al_p$ \,yield
\,$D_n \to +\infty$ \,and \,$E_n/D_n \to 0$ as $n \to \infty$,
from which one derives
$$
|P(f_{\al_1}(a_n), \dots ,f_{\al_p}(a_n))| \to +\infty \,\,\,\, (n \to \infty ),
$$
that contradicts $P(f_{\al_1}, \dots ,f_{\al_p}) \equiv 0$.
Hence $F := P(f_{\al_1}, \dots ,f_{\al_p}) \not \equiv 0$, which shows that \,$\cal A$ \,is freely generated.

\vskip .15cm

Our remaining task is to demonstrate that $F \in H_e(G)$.
Recall that $|F(a_n)| \to +\infty$ as $n \to \infty$. Assume, by way of contradiction, that $F \notin H_e(G)$.
Then there would exist some point $a \in G$ such that  $\rho (F,a) > {\rm dist} (a,\partial
G)$. Choose $r$ with
dist$(a,\partial G) < r < \rho (f,a)$. By the construction of
$\{a_n: \, n \in \N\}$, we can select a sequence $\{n_1 < n_2
< \cdots \} \subset \N$ for which $a_{n_k} \in G \cap B(a,r)$ $(k
\in \N )$. On the other hand, the sum $S(z)$ of the Taylor series
of $F$ with center $a$ is bounded on $B(a,r)$. But $S = F$ on $G
\cap B(a,r)$, so $S(a_{n_k}) = F(a_{n_k})$ $(k=1,2,...)$, which is absurd because $|F(a_{n_k})| \to +\infty$ as $k \to \infty$.
The proof is finished.
\end{proof}

Next, we extend parts (a) and (c) of Theorem \ref{Jordan-X-finitelength} by showing that the regularity of the domain
is enough to reach the same conclusions for the bigger class $H_{we}(G)$ (or for the same class
$H_e(G)$ if a little more is assumed). This will be carried out in Theorems \ref{H_e(G)Ainfty maximal dense-lineable} and
\ref{H_e(G)Ainfty spaceable}. We assume the Continuum Hypothesis (CH) in the following result.

\begin{theorem} \label{H_e(G)Ainfty maximal dense-lineable}
\begin{enumerate}
\item[\rm (a)] For any regular domain $G \subset \C$, the set \,$H_{we}(G) \cap A^\infty (G)$ \,is maximal dense-lineable in $A^\infty (G)$.
\item[\rm (b)] For any Jordan domain $G \subset \C$, the set \,$H_e(G) \cap A^\infty (G)$ \,is maximal dense-lineable in $A^\infty (G)$.
\end{enumerate}
\end{theorem}

\begin{proof}
Since (b) is a particular case of (a), we only must prove (b). To this end, and
according to Theorem \ref{Valdivia}, we select a nearly-Baire dense subspace \,$M \subset A^\infty (G)$ \,such
that $M \setminus \{0\} \subset H_e(G)$. Since dim$(A^\infty (G))$ $= \mathfrak{c}$,
it is enough to prove that every infinite dimensional nearly-Baire topological vector space $M$ must satisfy dim$(M) > \aleph_0$.
To do this, assume, by way of contradiction, that dim$(M) = \aleph_0$. Then there would exist a sequence of
vector subspaces
$$
X_1 \subset X_2 \subset \cdots \subset X_n \subset \cdots
$$
such that dim$(X_n) = n$ $(n=1,2,...)$ and $M = \bigcup_{n=1}^\infty X_n$.
As dim$(X_n) < \infty$, we have that each $X_n$ is closed and, trivially, balanced and sum-absorbing.
Consequently, some $X_{m}$ is a neighborhood of $0$, hence $M = X_{m}$, which is absurd because dim$(M) = \infty$.
\end{proof}

In order to face spaceability, the following two assertions (an elementary topological lemma and
a deep interpolation result by Valdivia \cite{Val0,Val2}, resp.) will be needed.
Recall that a topological space is called {\it perfect} whenever it lacks isolated points.

\begin{lemma} \label{LemmaT1perfect}
Assume that \,$X$ is a $T_1$, perfect, second countable topological space. Then from each dense sequence
\,$\{x_n\}_{n \ge 1}$ in \,$X$ one can extract infinitely many sequences \,$\{x_{n(k,j)}\}_{j \ge 1}$
$(k=1,2, \dots )$ \,such that each of them consists of different points and is dense in \,$X$, and they are pairwise disjoint.
\end{lemma}

\begin{proof}
Every nonempty open subset $U$ of $X$ is infinite. Indeed, if $U$ were finite, say $U = \{y_1, \dots ,y_p\}$
(with the $y_i$'s different), then $V := U \setminus \{y_1\} = \{y_2, \dots , y_p\}$ would be open (because any
singleton $\{y\}$ is closed, since $X$ is $T_1$). Therefore $V \setminus \{y_2\} = \{y_3, \dots ,y_p\}$
is open, and continuing this process we get after a finite number of steps that $\{y_p\}$ is open, which
is absurd because $X$ is perfect. Since $X$ is second countable, there exists a countable open basis $\{U_n\}_{n \ge 1}$.
Hence each member $U_n$ is infinite.

\vskip .15cm

Consider the following strict well-order in $\N \times \N$:
we say that $(l,s) < (k,j)$ if and only if either $l+s < k+l$ or $l+s=k+j$ but $l<k$. Then $(1,1)$ is the
least element of $\N \times \N$ and we have
$(1,1) < (1,2) < (2,1) < (1,3) < (2,2) < (3,1) < (1,4) < \cdots$. Since
$\{x_n\}_{n \ge 1}$ is dense, one can find $n(1,1) \in \N$ with $x_{n(1,1)} \in U_1$. Now, since $U_n \setminus F$
is open and nonempty for every $n$ and every finite set $F \subset X$, we may select, for each
$(k,j) > (1,1)$, an element $x_{n(k,j)} \in U_j \setminus \{x_{n(l,s)}: \, (l,s) < (k,j)\}$.
It is then plain that the sequences $\{x_{n(k,j)}\}_{j \ge 1}$
$(k=1,2, \dots )$ are dense in $X$ and pairwise disjoint, and each of them consists of different points.
\end{proof}

\begin{theorem} \label{Valdivia-2}
Let $G \subset \C$ be a regular domain. Then there is a dense subset $\{z_j: \, j \in \N\}$ in $\partial G$
consisting of different points such that,
for any of its arbitrary subsets $\{u_j: \, j \in \N\}$ and any infinite dimensional triangular matrix
$$
[a_{n+1,j}]_{j \ge n; \, n \in \N_0}
$$
of complex numbers, there is a function $f \in A_b^\infty (G)$ such that
$$
f^{(j)}(u_{n+1}) = a_{n+1,j} \quad (j \ge n; \, n \in \N_0).
$$
\end{theorem}

Let \,$\cal P$ \,denote the set of all polynomials in $z$. Of course, ${\cal P} \subset A^\infty (G)$.

\begin{theorem} \label{H_e(G)Ainfty spaceable}
\begin{itemize}
\item[\rm (a)] For any regular domain $G \subset \C$, the set \,$H_{we}(G) \cap A^\infty (G)$ \,is spaceable in $A^\infty (G)$.
\item[\rm (b)] For any Jordan domain $G \subset \C$, the set \,$H_e(G) \cap A^\infty (G)$ \,is spaceable in $A^\infty (G)$.
\end{itemize}
\end{theorem}

\begin{proof}
Again, it is enough to demonstrate (a).
Consider the sequence $\{z_j: \, j \in \N\} \subset \partial G$ whose existence is guaranteed by Theorem \ref{Valdivia-2}.
By Lemma \ref{LemmaT1perfect}, we can extract pairwise disjoint sequences
$\{z_{n(k,j)}\}_{j \ge 1}$ $(k = 1,2, \dots )$ such that each of them is still dense in $\partial G$.
Let $0! \cdot 0^0 := 1$.
According to Theorem \ref{Valdivia-2}, there exist functions $f_k \in A^\infty (G)$ $(k = 1,2, \dots )$
such that
$$
f_k^{(j)}(z_{n(k,l)}) = j!j^j \hbox{ \ for all \ } j \ge n(k,l)-1 \,\,\, (l \in \N ) \, \hbox{ and } \eqno (2)
$$
$$
f_k^{(j)}(z_{n(s,l)}) = 0 \hbox{ \ for all \ } s \ne k  \hbox{ \ and all \ } j \ge n(s,l)-1 \,\,\,(l \in \N ). \eqno (3)
$$
We have that the functions $f_k$ $(k \ge 1)$ are linearly independent. Indeed, if they were linearly dependent,
there would be $p \in \N$ as well as scalars $c_1, \dots ,c_p$ such that $c_p \ne 0$ and $F := \sum_{k=1}^p c_k f_k = 0$ on $\ovl{G}$.
Therefore $\sum_{k=1}^p c_k f_k^{(j)} = 0$ on $\partial G$ for every $j \ge 0$. In particular, if $N := n(p,1)$, we have
by (2) and (3) that
$$
0 = \sum_{k=0}^p c_k f_k^{(N)} (z_N) =  0 + c_p f_p^{(N)} (z_N) = c_p \, N!N^N,
$$
which is a contradiction.

\vskip .15cm

Let us define
$$
M := \hbox{span}\, \{f_k : \, k=1,2, \dots \}.
$$
Plainly, $M$ is an infinite dimensional vector subspace of $A^\infty (G)$. Fix $F \in M \setminus \{0\}$.
Then $F$ can be written as in the preceding paragraph, $F = \sum_{k=1}^p c_k f_k$, with $c_p \ne 0$.
In particular, by (2) and (3), we get \,
$$
F^{(j)}(z_{n(p,l)}) = c_p f_p^{(j)} (z_{n(p,l)}) = c_p \, j!j^j  \eqno (4)
$$
for every $l \in \N$ and every $j \ge n(p,l)$. Then for the radius of convergence of the associated Taylor series we have
$$
\rho (F,z_{n(p,l)}) = \left[ \limsup_{j \to \infty} \left| {F^{(j)}(z_{n(p,l)}) \over j!} \right|^{1/j} \right]^{-1} = 0
\hbox{ \ for all \ } l \in \N . \eqno (5)
$$
If $F \notin H_{we}(G)$ then there would be an open ball $B$ with $B \cap \partial G \ne \emptyset$ such that
$F$ extends holomorphically on $B$. Due to the density of $\{z_{n(p,l)}\}_{l \ge 1}$, one can select $l \in \N$
with $z_{n(p,l)} \in B$, which is impossible by (5). Thus $M \setminus \{0\} \subset H_{we}(G) \cap A^\infty (G)$.

\vskip .15cm

Now, consider the space
$$
X := \overline{M} = \overline{\hbox{span}} \, \{f_k : \, k=1,2, \dots \},
$$
where the closure is taken in $A^\infty (G)$.
Since $A^\infty (G)$ is a Fr\'echet space, we get that $X$ is a Fr\'echet space under the inherited topology.
Suppose that $F \in X$. Then there exists a sequence $\{F_\nu = \sum_{k=1}^\infty \la_{\nu ,k} f_k\}_{\nu \ge 1} \subset M$
such that, for every $j \in \N$, $F_{\nu}^{(j)} \to F^{(j)}$ $(\nu \to \infty )$ uniformly on every compact
subset of \,$\ovl{G}$; here the coefficients $\la_{\nu,k}$ are complex numbers such that, for every $\nu$,
$\la_{\nu ,k} = 0$ provided that $k$ is large enough. Therefore
\,$\sum_{k=1}^\infty \la_{\nu ,k} f_k^{(j)} (z_{n(m,l)}) \to F^{(j)}(z_{n(m,l)})$ as $\nu \to \infty$,
for all $j \in \N_0$ and all $m,l \in \N$. From (2) and (3), we get
$$
\la_{\nu ,m} j!j^j \to F^{(j)}(z_{n(m,l)}) \,\,\, (\nu \to \infty ) \,\,\, \hbox{for all } m,l \in \N \hbox{ and all } j \ge n(m,l).
$$
Then $\la_{\nu ,m} \to F^{(j)}(z_{n(m,l)})/j!j^j$ as $\nu \to \infty$. Now the uniqueness of the limit
leads us to the existence, for each $m \in \N$, of a constant $K_m \in \C$ such that
$$
F^{(j)}(z_{n(m,l)})= K_m \, j!j^j \, \hbox{ provided that } j \ge n(m,l). \eqno (6)
$$
Assume now that $F \in X \setminus H_{we}(G)$. Then we have again at our disposal an open ball $B$ with
$B \cap \partial G \ne \emptyset$ such that $F$ extends holomorphically on $B \cup G$.
Fix any $m \in \N$. If $K_m \ne 0$
then (6) entails that $\rho (F,z_{n(m,l)}) = 0$ for each $l$, which is impossible because, by density,
there is $l$ with $z_{n(m,l)} \in B$. Consequently, $K_m = 0$ and, again by the density of
$\{z_{n(m,l)}\}_{l \ge 1}$, (6) implies that $F^{(j)}(w) = 0$ (if $j$ is large enough) for at least one point $w \in B$.
The Identity Principle tells us that $F$ is a polynomial. Hence $X \setminus H_e(G) \subset {\cal P}$.
Let $Y := \ovl{{\cal P} \cap X}$, which is a closed linear subspace of $X$. Observe that
$$
H_e(G) \cap A^\infty (G) \supset X \setminus {\cal P} = X \setminus ({\cal P} \cap X) \supset X \setminus Y. \eqno (7)
$$
Suppose that $F \in {\cal P} \cap X$. Then (6) implies that the corresponding constants $K_m$ are all $0$, so
$$
F^{(j)}(z_{n(m,l)})= 0 \,\, \hbox{ for all } j \ge n(m,l) \,\,\,\, (m,l=1,2, \dots ) \eqno (8)
$$
If now $F \in Y$ then $F$ is a limit in $A^\infty (G)$ of a sequence of
functions each of them satisfying (8). Hence $F$ also satisfies
(8), which is inconsistent with (4) if, in addition, $F \in M$, except that $F = 0$. In other words,
$Y \cap M = \{0\}$. But $M$ is an infinite dimensional vector space contained in $X$. Thus $Y$ has infinite
codimension in $X$. From Wilansky--Kalton's Theorem \ref{Wilansky-Kalton} one derives that $X \setminus Y$
is spaceable in $X$. Since $X$ is closed in $A^\infty (G)$, a subset of $X$ is closed in $X$ if and only if
it is closed in $A^\infty (G)$. This fact and (7) entail the desired spaceability of $H_e(G) \cap A^\infty (G)$
in $A^\infty (G)$.
\end{proof}

\begin{remark}
{\rm Observe that, by using Theorem \ref{Valdivia-2}, the same proof above works to show
the spaceability of \,$H_{we}(G) \cap A_b^\infty (G)$ \,in
\,$A_b^\infty (G)$, whenever \,$G \subset \C$ \,is regular.}
\end{remark}

Algebrability of $H_{we}(G)$ inside $A^\infty (G)$ can also be asserted.
This complements the final part of Theorem \ref{AronGarciaMaestre} and
will be shown in Theorem \ref{H_e(G)Ainfty algebrable} below, but prior to it we
need the next auxiliary statement, which is probably well known. For the sake of completeness,
we provide an elementary proof of this statement. By \,$A^1(G)$ \,it is denoted the space
of functions \,$f \in H(G)$ \,such that $f$ and $f'$ extend continuously to \,$\ovl{G}$.

\begin{lemma} \label{Lemma-composition}
Suppose that $G \subset \C$ is a domain such that
\,$\partial G$ does not contain isolated points. Let \,$f \in H_{we}(G) \cap A^1(G)$ \,and \,$\varphi$
be a nonconstant entire function. Then \,$\varphi \circ f \in H_{we}(G)$.
\end{lemma}

\begin{proof}
Assume, by way of contradiction, that \,$F := \varphi \circ f \not\in H_{we}(G)$.
Then there is an open ball $B$ centered at some point $z_0 \in \partial G$ as well as a
function $\widetilde{F} \in H(G \cup B)$ such that $\widetilde{F} = F$ on $G$.
Let $B_1 \subset B$ be any closed ball centered at $z_0$.
If $\widetilde{F}'(z) = 0$ for all $z \in  B \cap \partial G$ then, since $\partial G$ is perfect,
we would have \,$\widetilde{F}'= 0$ \,on a subset of $B$ having some accumulation point in $B$
(namely, on $B_1 \cap \partial G$). By the Analytic Continuation Principle, $\widetilde{F}'= 0$ on $B$,
hence $\widetilde{F}$ is constant on $B$. By the same Principle, and since $\widetilde{F}=F$ on $G$,
we get $F$ = constant in $G$. Since $f \in H_{we}(G)$, $f$ is not constant, so $f(G)$ is open by the Open Mapping Theorem for
analytic functions. Therefore $\varphi$ is constant on the open set $f(G)$, and a third application of the Analytic
Continuation Principle yields $\varphi$ = constant, which is absurd. Then there must be $z_1 \in B \cap \partial G$
with $\widetilde{F}'(z_1) \ne 0$. From the Local Representation Theorem (see e.g.~\cite{Alf}) we derive
the existence of an open ball $B_2 \subset B$ centered at $z_1$ and of a domain $W$ with $W \ni \widetilde{F}(z_1) = \varphi (f(z_1))$
(recall that $f$ extends continuously to $\partial G$) such that $\widetilde{F} : B_2 \to W$ is bijective.
In particular, $0 \ne \widetilde{F}'(z_1) = \varphi ' (f(z_1)) f'(z_1)$, where in the last equality the hypothesis
$f \in A^1(G)$ has been used. Thus $\varphi ' (f(z_1)) \ne 0$ and, again by the Local Representation Theorem,
there are an open ball $B_0$ centered at $f(z_1)$ and a domain $V$ with $\varphi (f(z_1)) \in V \subset W$
such that $\varphi : B_0 \to V$ is bijective. Then $(\widetilde{F}|_W)^{-1}(V)$ is a domain satisfying
$z_1 \in (\widetilde{F}|_W)^{-1}(V) \subset B$. Consequently, $G \cup (\widetilde{F}|_W)^{-1}(V)$
is a domain containing \,$G$ \,strictly
and the function $\widetilde{f}: G \cup (\widetilde{F}|_W)^{-1}(V) \to \C$ given by
$$
\widetilde{f} (z) = \left\{
\begin{array}{ll}
f(z) & \mbox{if } z \in G \\
\varphi^{-1}(\widetilde{F}(z)) & \mbox{if } z \in (\widetilde{F}|_W)^{-1}(V)
\end{array}
\right.
$$
is well defined, holomorphic and extends $f$. This contradicts our assumption that $f \in H_{we}(G)$ and the proof is finished.
\end{proof}

\begin{remark}
{\rm The conclusion of the last lemma fails if no condition is imposed on $f$. For instance,
if \,$G = \C \setminus (-\infty ,0]$, $f$ is the principal branch of \,$\log z$ \,and
\,$\varphi = \exp$, then \,$f \in H_{we}(G)$ \,but \,$\varphi \circ f =$ Identity $\not\in H_{we}(G)$.}
\end{remark}

\begin{theorem} \label{H_e(G)Ainfty algebrable}
\begin{itemize}
\item[\rm (a)] Let \,$G \subset \C$ \,be a domain such that \,$\partial G$ \,lacks isolated points, and let
\,$X$ \,be an algebra of functions with \,$X \subset A^1(G)$
\,that is stable under composition with entire functions, that is,
$$
f \in X \hbox{ \ and \ } \varphi \in H(\C ) \,\, \Longrightarrow \,\, \varphi \circ f \in X.
$$
Then the family \,$H_{we}(G) \cap \,X$ \,is either empty or strongly $\mathfrak{c}$-algebrable.
\item[\rm (b)] For any regular domain $G \subset \C$, the set $H_{we}(G) \cap A^\infty (G)$ is strongly $\mathfrak{c}$-algebrable.
For any Jordan domain $G \subset \C$, the set $H_{e}(G) \cap A^\infty (G)$ is strongly $\mathfrak{c}$-algebrable.
\end{itemize}
\end{theorem}

\begin{proof}
Evidently, (b) is a special case of (a), because \,$A^\infty (G)$ \,is an algebra contained in \,$A^1(G)$ \,that is stable
under composition with members of \,$H(\C )$ and, $G$ being regular, we have \,$H_{we}(G) \cap A^\infty (G) \ne \emptyset$
\,and \,$\partial G$ \,lacks isolated points. Therefore, our goal is to prove (a).

\vskip .15cm

To this end, suppose that \,$f \in H_{we}(G) \cap \,X$. Let $\Omega := G$, ${\cal F} :=
H_{we}(G) \cap \,X$. Since $f$ is nonconstant, the set $f(\Omega )$ is open, so uncountable. According to our assumptions and Lemma \ref{Lemma-composition}, we have that \,$\varphi \circ f \in {\cal F}$ \,for every $\varphi \in {\cal E}$, the family of
exponential-like functions. Finally, thanks to Proposition \ref{exponentials}, the set \,${\cal F}$ \,is strongly $\mathfrak{c}$-algebrable,
as desired.
\end{proof}

\begin{remark}
{\rm We have used several times the fact that \,$H_e(G) = H_{we}(G)$ \,if $G$ is a Jordan domain.
More generally, it is easy to see that \,$H_e(G) = H_{we}(G)$ \,if the domain \,$G$ \,satisfies the following property:
{\it
\begin{enumerate}
\item[\rm (P)] For every open ball \,$B$ \,with \,$B \not\subset G$ \,and \,$B \cap G \ne \emptyset$, and every connected component
\,$A$ of \,$B \cap G$, there exists an open ball \,$B_0$ \,satisfying \,$B \supset B_0 \not \subset G$ \,and \,$A \supset B_0 \cap G \ne \emptyset$.
\end{enumerate}}

\noindent Then we can replace ``For any Jordan domain'' by ``For any regular domain satisfying (P)'' in part (b) of Theorems \ref{H_e(G)Ainfty maximal dense-lineable}, \ref{H_e(G)Ainfty spaceable} and \ref{H_e(G)Ainfty algebrable}. Note that, for instance,
\,$\Omega_1 := \{z: \, |z-1| < 1$ and $|z-(1/2)| > 1/2\}$ \,and
\,$\Omega_2 := \{z=x+iy: \, x < 0 \hbox{ or } [x = 0 \hbox{ and } y < 0] \hbox{ or } [x > 0 \hbox{ and } y < 1 + \sin {1 \over x}]\}$
are regular non-Jordan domains such that (P) holds for \,$\Omega_1$ \,but not for \,$\Omega_2$.}
\end{remark}

Our next result establishes, for an arbitrary complex Banach space, that certain ``good shape'' of
the domain guarantees the {\it dense} lineability.
This complements Theorem \ref{Alves}.
Recall that in $H(G)$ we are considering the to\-po\-lo\-gy of uniform convergence in compacta
(see \cite{Chae} or \cite{Din} for a description of this topology and other ones if $E$ is infinite dimensional).
Recall also that a subset $A$ of a vector space is called {\it balanced} provided that
$\la x \in A$ whenever $x \in A$ and $|\la | \le 1$.

\begin{theorem} \label{H_e(G)in infinite dimension}
Suppose that \,$G$ is a domain of existence of a separable complex Banach space $E$
such that there is $x_0 \in G$ such that $G - x_0$ is balanced.
Then \,$H_e(G)$ is maximal dense-lineable in \,$H(G)$.
\end{theorem}

\begin{proof}
Since \,$f \in H_e(G) =: A$ \,if and only if \,$f(\cdot + x_0) \in H_e(G - x_0)$, we can suppose that $x_0 = 0 \in G$ and $G$ is balanced.
Under this assumption, the Taylor series centered at $0$ of each $f \in H(G)$ converges to $f$ uniformly on compacta
(see \cite[Proposition 3.36]{Din}). Consequently, the set \,$B$ \,of the restrictions to \,$G$ \,of the (continuous) polynomials in \,$E$
\,is dense in \,$H(G)$. Hence \,$B$ \,is dense-lineable, because it is a vector space.
Notice that, since $G$ is separable (because $E$ is), the space $C(G)$ of complex continuous functions on $G$ has cardinality $\mathfrak{c}$ which,
together with the fact $C(E) \supset H(G) =: X$, implies ${\rm dim} (X) = \mathfrak{c} = {\rm card}\,(X)$.

\vskip .15cm

On one hand, we have by Theorem \ref{Alves} that \,$A$ \,is $\mathfrak{c}$-lineable.
On the other hand, it is evident that if \,$f \in H_e(G)$ \,and \,$g \in H(E)$ \,then \,$f + g \in H_e(G)$.
In particular, $A + B \subset A$. Trivially, $A \cap B = \emptyset$. Now, observe that the collection
\,${\cal B}$ \,of all sets of the form
$$
V(f,K,\ve ) = \{g \in X: \, f \in X, \, \ve > 0, \,K \hbox{ compact } \subset G\}
$$
\,is an open basis for the topology of \,$H(G)$. Note also that,
as $G$ is separable, every compact subset $K$ is closed and separable, so $K$ is the closure of some countable subset
of $G$. Since ${\rm card}\,(G) = \mathfrak{c}$, the collection of the countable subsets of $G$ has cardinality $\mathfrak{c}$,
and therefore the same holds for the collection \,$\cal K$ \,of all compact subsets of $G$. Hence \,${\rm card}\,(X) = \mathfrak{c} =
{\rm card}\,(0,+\infty ) = {\rm card}\,({\cal K})$, so \,${\rm card}\,({\cal B}) = \mathfrak{c}$.
An application of Theorem \ref{maximal dense-lineable} (with $\al = \mathfrak{c}$) concludes the proof.
\end{proof}

\noindent {\bf Question.} Let $E$ be a separable complex Banach space.
Under appropriate conditions, is \,$H_e(G)$ {\it densely} \,algebrable in \,$H(G)$?
Is it {\it $\mathfrak{c}$-algebrable}\,?

\vskip .15cm

We finish this paper by establishing that, for any {\it finite dimensional} \,domain of existence, the conclusion of
Theorem \ref{H_e(G)in infinite dimension} always holds and the second part of the last question
has always a positive answer. This complements Theorem \ref{AronGarciaMaestre}. In the second part, the CH will be assumed.

\begin{theorem}
Let $N \in \N$ and $G \subset \C^N$ be a domain of existence. Then \,$H_e(G)$ \,is maximal dense-lineable in \,$H(G)$
\,and $\mathfrak{c}$-algebrable.
\end{theorem}

\begin{proof}
The maximal dense-lineability of $H_e(G)$ can be shown exactly as in the proof of Theorem \ref{H_e(G)Ainfty maximal dense-lineable},
by using the existence of a dense nearly-Baire subspace (see the paragraph following Theorem \ref{AronGarciaMaestre}).
As for $\mathfrak{c}$-algebrability, suppose, by way of contradiction, that any algebra in $H_e(G) \cup \{0\}$ is coun\-ta\-bly
generated. In particular, the closed algebra $\cal A$ given in the conclusion of Theorem \ref{AronGarciaMaestre}
contains a countable set $\{f_n: \, n=1,2,...\}$ such that eve\-ry $f \in {\cal A}$
is a linear combination of products of the form $f_{i_1}^{m_1} \cdots f_{i_p}^{m_p}$ with $m_1, \dots ,m_p \in \N$ ($p \in \N$).
Observe that there are countably many such products.
But ${\cal A}$ was infinite dimensional when considered as a vector space. Since ${\cal A}$ is closed in the F-space $H(G)$,
the space ${\cal A}$ is a separable infinite dimensional F-space. A standard application of Baire's category theorem yields \,${\rm dim}\,({\cal A}) = \mathfrak{c}$, which contradicts the fact that $\cal A$ is a countably ge\-ne\-ra\-ted vector space.
The proof is finished.
\end{proof}

\noindent {\bf Acknowledgements.} The authors have been partially supported by the Plan
Andaluz de Investigaci\'on de la Junta de Andaluc\'{\i}a FQM-127
Grant P08-FQM-03543 and by MEC Grant MTM2012-34847-C02-01.

\bigskip

{\footnotesize  

} 

\end{document}